\documentclass[12pt]{article}

\usepackage[letterpaper,margin=1in]{geometry}

\usepackage[utf8]{inputenc}
\usepackage[T1]{fontenc}
\usepackage{lmodern}

\usepackage{mathtools,amssymb,amsthm}
\usepackage{authblk}
\usepackage{enumitem}
\usepackage{booktabs}
\usepackage{longtable}
\usepackage{array}
 
\usepackage{bbm}

\usepackage[hidelinks]{hyperref}

\numberwithin{equation}{section}

\newtheorem{theorem}{Theorem}[section]

\newtheorem{lemma}[theorem]{Lemma}

\newtheorem{problem}[theorem]{Problem}

\theoremstyle{remark}

\DeclareMathOperator{\ex}{ex}

\newcommand{\F}{\mathbb F}
\newcommand{\PG}{\operatorname{PG}}

\title{{\Large\bf Explicit thresholds in a generalized Tur\'an problem for \(K_{3,t}\)-free graphs}}
\author[1]{Jianfeng Hou\thanks{Email: \texttt{jfhou@fzu.edu.cn}}}
\author[1]{Caiyun Hu\thanks{Email: \texttt{hucaiyun.fzu@gmail.com}}}
\author[1]{Hezhi Wang\thanks{Email: \texttt{wanghezhi128@gmail.com}}}
\affil[1]{\small Center for Discrete Mathematics, Fuzhou University, Fuzhou, China}
\date{\today}

\begin{document}
\maketitle

\begin{abstract}

For graphs $F$ and $H$, let $\ex(n,F,H)$ denote the maximum number of
copies of $F$ in an $n$-vertex $H$-free graph. Janzer, Longbrake and
Yepremyan recently proved that, for fixed $3<a\le b$ and sufficiently large
$t$,
\[
   \ex(n,K_{a,b},K_{3,t})=\Theta(n^3). 
\]
We make their threshold explicit, showing that this conclusion holds for all
$t\ge \tau(b):=2\max\{3,\lceil b/2\rceil\}+1.$
In particular, for every even $b\ge 6$, this matches the necessary threshold
$t=b+1$. The main new ingredient is an explicit finite-field point set whose plane sections are controlled directly, rather than through a general bounded-complexity algebraic lemma. This direct line-and-conic section analysis gives the required \(K_{3,t}\)-freeness while preserving many coplanar \(b\)-element subsets.

\end{abstract}



\section{Introduction}

Let $H$ be a fixed graph. A graph $G$ is $H$-free if $G$ does not  contain a copy of $H$ as a subgraph. The Tur\'an number of $H$, denoted by $\ex(n,H)$, is the largest number of edges in an $n$-vertex $H$-free graph. Determining $\ex(n,H)$ is  a basic  problem in extremal graph theory.  The classic Erd\H{o}s--Stone--Simonovits theorem gives an asymptotic for $\ex(n,H)$ when $H$ is not bipartite.

For bipartite $H$, much less is known. The basic test case is
$H=K_{s,t}$. For $t\ge s\ge 2$, the K\H{o}v\'ari--S\'os--Tur\'an theorem
gives
\begin{equation*}
   \ex(n,K_{s,t})=O_{s,t}(n^{2-1/s}),
\end{equation*}
and a central question is whether the exponent $2-1/s$ is best possible.
Equivalently, one asks whether there exist $K_{s,t}$-free graphs with
$\Omega_{s,t}(n^{2-1/s})$ edges; this is the classical Zarankiewicz problem
\cite{KST}. Finite-geometric and algebraic constructions establish sharpness
in several cases \cite{ERS,Brown,Furedi}. More generally,
Koll\'ar--R\'onyai--Szab\'o proved sharpness for $t\ge s!+1$ using norm
graphs \cite{KRS}, and Alon--R\'onyai--Szab\'o improved this to
$t\ge (s-1)!+1$ via projective norm graphs \cite{ARST}. A breakthrough of
Bukh introduced a new random algebraic construction, which establishes
sharpness already for $t\ge C^s$, where $C$ is an absolute constant
\cite{BukhBicliques}.

This paper studies a generalized Tur\'an problem for complete bipartite graphs. For graphs \(F\) and \(H\), let \(\ex(n,F,H)\) be the maximum number of copies of \(F\) in an \(n\)-vertex \(H\)-free graph. This parameter was introduced systematically by Alon and Shikhelman \cite{AS}; the ordinary Tur\'an number is the special case \(F=K_2\). Given positive integers $a\le b$ and $2\le s\le t$, we are interested in the asymptotic growth of the function $\ex(n,K_{a,b},K_{s,t})$. Clearly, the question is nontrivial only when $a<s$ or $b<t$.

Note that $\ex(n,K_{a,b},K_{s,t})$ is well understood in several ranges with
$a\le s$. The complementary regime
\[
   2\le s<a\le b<t
\]
requires different constructions. As observed in \cite{PTY,JLY}, in this range
there is a natural upper bound obtained by a simple counting argument, namely
\begin{equation}\label{Mian-bound-general-tran-bi}
   \ex(n,K_{a,b},K_{s,t})=O(n^s).
\end{equation}
A natural problem is whether the bound in
\eqref{Mian-bound-general-tran-bi} is tight. Most known results focus on small
values of $s$. Pohoata, Tidor and Yu \cite{PTY} resolved the first case of
this type by proving for every fixed $t\ge 3$, $K_{2,t+1}$-free graphs
with $\Omega_t(n^2)$ copies of $K_{t,t}$. Very recently, Janzer, Longbrake
and Yepremyan \cite{JLY} considered the cases $s\in\{2,3\}$ and proved the
following.
\begin{theorem}
For \(s\in\{2,3\}\) and \(s<a\le b\), there is a constant \(t_0\) such that
\[
   \ex(n,K_{a,b},K_{s,t})=\Theta(n^s)
\]
for all \(t\ge t_0\). 
\end{theorem}

In the case \(s=3\) \cite{JLY}, their construction introduced a projective point-hyperplane incidence framework in dimension five. The largeness condition on \(t\) comes from their algebraic input. Their bounded-complexity variety lemma gives a dichotomy for each relevant plane section: either the section has fewer than \(T\) points, where \(T\) is a constant depending only on the fixed algebraic complexity, or it is a large section and is placed in the exceptional family \cite[Lemma 3.1]{JLY}. These exceptional sections are deleted. After the deletion, a copy of \(K_{3,t}\) could only come from a remaining section with at least \(t\) points, so the construction is guaranteed to be \(K_{3,t}\)-free only when \(t>T\). Since this constant \(T\) is not explicit as a function of \(b\), their theorem yields only a qualitative threshold \(t_0\).

Our main contribution is to turn this qualitative threshold into an explicit one. For \(s=3\), we prove that \(t\ge \tau(b)\) suffices, where \(\tau(b)\) is defined below; when \(b\ge 6\) is even, this gives the best possible threshold \(t=b+1\). The key new ingredient is an explicit finite-field point set with a sharp plane-section bound. We keep the projective incidence and deletion framework of Janzer, Longbrake and Yepremyan, but replace their general bounded-complexity algebraic input by a concrete construction whose plane sections can be controlled directly. This direct control is what makes the threshold explicit. More precisely, the point set is built from an anisotropic quadratic surface, parabolic fibres, and a low-degree random polynomial; the relevant sections reduce to affine lines and nonsingular conics, where elementary interpolation gives the required bound while still leaving many planes rich enough to create the desired copies of \(K_{a,b}\).

For an integer $b\ge 4$, define
\[
   \tau(b):=2\max\{3,\lceil b/2\rceil\}+1.
\]
Equivalently,
\[
\tau(b)=
\begin{cases}
7, & 4\le b\le 6,\\
b+1, & b\ge 8\text{ and }b\text{ is even},\\
b+2, & b\ge 7\text{ and }b\text{ is odd}.
\end{cases}
\]

\begin{theorem}\label{thm:main}
Let \(3<a\le b\) and \(t\) be fixed integers. If \(t\ge \tau(b)\), then
\[
   \ex(n,K_{a,b},K_{3,t})=\Theta_{a,b,t}(n^3).
\]
\end{theorem}
It is worth noting that $t\ge b+1$ is essential.  Thus Theorem \ref{thm:main} reaches the necessary threshold for every even \(b\ge 6\).

The rest of the paper is organized as follows. Section 2 constructs the algebraic point set
and proves the geometric input needed later, including the absence of three collinear points, the explicit plane-section bound, and the abundance of rich non-exceptional planes.
Section 3 converts this input into a projective point-hyperplane incidence graph and proves
the graph-theoretic lower bound, from which Theorem~\ref{thm:main} follows. Section 4 records the
remaining odd-case threshold question.

\section{The algebraic point set}

This section supplies the algebraic core of the construction. We define the finite-field point set, prove that it has no three collinear points, analyze its intersections with affine planes, and choose a polynomial for which there are few exceptional planes but many rich non-exceptional planes. The output is Lemma \ref{prop:algebraic-set}, which is the geometric input for the incidence construction.

Throughout this section \(q\) is an odd prime power, \(\alpha\in\mathbb F_q\) is a
nonsquare, and points of \(\mathbb F_q^5\) are written as
\((x,y,w,z,u)\). Let
\[
\pi(x,y,w,z,u)=(x,y,w)
\]
be the projection onto the first three coordinates. Fix \(r\ge 3\), and let
\(\mathcal F_r\) be the vector space of polynomials
\(f(X,Y,W)\in\mathbb F_q[X,Y,W]\) of total degree at most \(r\). For
\(f\in\mathcal F_r\), define
\[
S_f=
\{(x,y,x^2-\alpha y^2,z,z^2+f(x,y,x^2-\alpha y^2)):
x,y,z\in\mathbb F_q\}\subset \mathbb F_q^5 .
\]
The parametrization is injective in \((x,y,z)\), so \(|S_f|=q^3\). We write
\[
\Sigma:=\pi(S_f)
=
\{(x,y,x^2-\alpha y^2):x,y\in\mathbb F_q\}\subset \mathbb F_q^3
\]
for the first-three-coordinate projection of \(S_f\).

\subsection{No three collinear points}
We begin with two elementary lemmas needed to exclude collinear triples in \(S_f\). The first concerns the projection of \(S_f\) onto its first three coordinates and uses the fact that \(X^2-\alpha Y^2\) has no nonzero zero over \(\F_q\).
\begin{lemma}\label{lem:surface-lines}
The set \(\Sigma\) contains no three collinear points.
\end{lemma}

\begin{proof}
Let \(\lambda\) be an affine line in \(\F_q^3\), written as
\[
   (x,y,w)=(x_0,y_0,w_0)+s(u,v,w_1),\qquad s\in\F_q.
\]
The points of \(\lambda\cap\Sigma\) are the solutions of
\[
   (x_0+su)^2-\alpha (y_0+sv)^2=w_0+sw_1.
\]
If \((u,v)\ne(0,0)\), the quadratic coefficient is \(u^2-\alpha v^2\), which is nonzero because \(\alpha\) is a nonsquare. After moving all terms to one side, we obtain a nonzero quadratic polynomial in \(s\). Since a nonzero polynomial of degree two over a field has at most two roots, the equation has at most two solutions in \(s\). If \((u,v)=(0,0)\), then \(w_1\ne0\), since \((u,v,w_1)\) is a nonzero direction vector. In this case the first two coordinates are fixed, equal to \(x_0\) and \(y_0\). Therefore the intersection condition becomes
\[
   w_0+sw_1=x_0^2-\alpha y_0^2.
\]
This is a nonzero linear equation in \(s\), so it has exactly one solution. Hence \(\lambda\cap\Sigma\) has at most one point in this case. In both cases, \(\lambda\cap\Sigma\) has at most two points. Thus \(\Sigma\) contains no three collinear points.
\end{proof}

After controlling lines on the projected surface, we shall also need the same conclusion inside each fibre of \(\pi\). There, the relevant set is the parabola with last two coordinates \((z,z^2+C)\).
\begin{lemma}\label{lem:parabola-lines}
For every \(C\in\F_q\), the parabola
\[
   \{(z,z^2+C):z\in\F_q\}\subset\F_q^2
\]
contains no three collinear points.
\end{lemma}

\begin{proof}
Let \(\ell\) be an affine line in the \((z,u)\)-plane. If \(\ell\) is nonvertical, then it has an equation
\[
   u=az+b
\]
for some \(a,b\in\F_q\). Its intersection with the parabola is given by
\[
   z^2+C=az+b.
\]
After moving all terms to one side, this is a nonzero quadratic equation in \(z\). Hence it has at most two solutions over \(\F_q\). If \(\ell\) is vertical, then it has an equation \(z=z_0\). The parabola contains exactly one point with this first coordinate, namely \((z_0,z_0^2+C)\). Thus \(\ell\) meets the parabola in at most one point. In both cases, \(\ell\) meets the parabola in at most two points. Therefore the parabola contains no three collinear points.
\end{proof}
We now combine the two lemmas to prove that \(S_f\) itself has no three collinear points. The projection to the first three coordinates excludes the case in which the three projected points are distinct, while the parabolic fibres exclude the remaining case.
\begin{lemma}\label{lem:S-no-three}
For every \(f\in\mathcal F_r\), the set \(S_f\) contains no three collinear points.
\end{lemma}

\begin{proof}
Assume that three distinct points \(P_1,P_2,P_3\in S_f\) are collinear. Let \(\Lambda=\{P_0+sD:s\in\F_q\}\), with \(D\ne0\), be the
affine line containing them, and write \(P_i=P_0+s_iD\). Since
\[
   \pi(P_0+sD)=\pi(P_0)+s\pi(D),
\]
the image \(\pi(\Lambda)\) is a point when \(\pi(D)=0\), and otherwise is
an affine line with direction \(\pi(D)\).

For each \(i\), we have \(\pi(P_i)\in\Sigma\cap\pi(\Lambda)\). If  \(\pi(P_1),\pi(P_2),\pi(P_3)\) are pairwise  distinct, then there are three
distinct collinear points of \(\Sigma\), contradicting Lemma
\ref{lem:surface-lines}. Hence \(\pi(P_i)=\pi(P_j)\) for some \(i\ne j\).
Then
\[
   0=\pi(P_i)-\pi(P_j)=(s_i-s_j)\pi(D).
\]
It follows from \(s_i\ne s_j\) that \(\pi(D)=0\). Thus every point of
\(\Lambda\) has the same first three coordinates, and in particular
\[
   \pi(P_1)=\pi(P_2)=\pi(P_3).
\]

Let this common projection be \((x,y,x^2-\alpha y^2)\). Then, there exist pairwise distinct 
\(z_1,z_2,z_3\in\F_q\), such that 
\[
   P_i=\bigl(x,y,x^2-\alpha y^2,z_i,z_i^2+C\bigr),
   \qquad C=f(x,y,x^2-\alpha y^2).
\]
Since \(D\ne0\) and \(\pi(D)=0\), let  \(D=(0,0,0,d_4,d_5)\) with 
\((d_4,d_5)\ne(0,0)\). Therefore the projection of \(\Lambda\) to the last
two coordinates is
\[
   \{(a,b)+s(d_4,d_5):s\in\F_q\}
\]
for some \((a,b)\in\F_q^2\), hence an affine line in \(\F_q^2\). The three
distinct points \((z_i,z_i^2+C)\), \(i=1,2,3\), lie on this line and on the
parabola \(\{(z,z^2+C):z\in\F_q\}\), contradicting Lemma
\ref{lem:parabola-lines}. This proves the lemma.
\end{proof}

\subsection{Plane sections}

We now consider an affine plane \(H\subset\F_q^5\). Since the first three coordinates of \(S_f\) lie on \(\Sigma\), and the points over each point of \(\Sigma\) form a parabola in the last two coordinates, the analysis of \(H\cap S_f\) is organized according to \(\dim\pi(H)\). We begin with the cases \(\dim\pi(H)\le1\).

\begin{lemma}\label{lem:small-projection}
For every \(f\in\mathcal F_r\), the following hold.
\begin{enumerate}
\item[\textup{(i)}] If \(\dim\pi(H)=0\) and \(H\cap S_f\ne\emptyset\), then
\[
   H=Q_{x,y}:=\{(x,y,x^2-\alpha y^2,u,v):u,v\in\F_q\}
\]
for some \(x,y\in\F_q\), and \(|H\cap S_f|=q\). In particular, there are exactly \(q^2\) such planes.
\item[\textup{(ii)}] If \(\dim\pi(H)=1\), then \(|H\cap S_f|\le 4\).
\end{enumerate}
\end{lemma}

\begin{proof}
First,  assume that \(\dim\pi(H)=0\) and \(H\cap S_f\ne\emptyset\). Choose \(P\in H\cap S_f\). Then \(\pi(P)=(x,y,x^2-\alpha y^2)\) for some \(x,y\in\F_q\). Since \(\pi(H)\) consists of a single point,
\[
   H\subseteq Q_{x,y}:=
   \{(x,y,x^2-\alpha y^2,s,t):s,t\in\F_q\}.
\]
Since both \(H\) and \(Q_{x,y}\) are affine planes and \(H\subseteq Q_{x,y}\), we have
\(H=Q_{x,y}\). Moreover, \(Q_{x,y}\) is determined by its fixed first two coordinates, so there are exactly \(q^2\) such planes.

Recall that a point of
\(S_f\) has the form
\[
   (a,b,a^2-\alpha b^2,z,z^2+f(a,b,a^2-\alpha b^2)).
\]
Membership in \(Q_{x,y}\) yields that \(a=x\) and \(b=y\). Hence, with
\(C=f(x,y,x^2-\alpha y^2)\),
\[
   H\cap S_f
   =
   \{(x,y,x^2-\alpha y^2,z,z^2+C):z\in\F_q\}.
\]
This implies that \(|H\cap S_f|=q\).

Now, assume  that \(\dim\pi(H)=1\), and set \(L=\pi(H)\). Then \(L\) is an affine line, and every point of \(H\cap S_f\) projects to \(L\cap\Sigma\).  If \(H\cap S_f=\emptyset\), there is nothing to prove; otherwise fix \(A=(x,y,x^2-\alpha y^2)\in\pi(H\cap S_f)\). Let \(T\) be the direction space of \(H\). Since \(\pi(T)\) is one-dimensional, choose a basis \(U,V\) of \(T\) with \(\pi(U)\ne0\) and \(\pi(V)=0\), and write
\[
   H=\{P_0+sU+tV:s,t\in\F_q\}.
\]
Then \(L=\{\pi(P_0)+s\pi(U):s\in\F_q\}\). Since \(A\in L\), there is a
unique \(s_0\in\F_q\) such that \(A=\pi(P_0)+s_0\pi(U)\), and therefore
\[
   H\cap\pi^{-1}(A)=\{P_0+s_0U+tV:t\in\F_q\}.
\]
Thus the points of \(H\) above \(A\) form an affine line in the fibre \(\pi^{-1}(A)\).

The projection to the last two coordinates identifies the fibre \(\pi^{-1}(A)\) with \(\F_q^2\). Since \(\pi(V)=0\) and \(V\ne0\), this projection sends \(H\cap\pi^{-1}(A)\) to an affine line in \(\F_q^2\). On the other hand, with \(C=f(x,y,x^2-\alpha y^2)\),
\[
   S_f\cap\pi^{-1}(A)
   =
   \{(x,y,x^2-\alpha y^2,z,z^2+C):z\in\F_q\},
\]
which becomes the parabola \(\{(z,z^2+C):z\in\F_q\}\). By Lemma \ref{lem:parabola-lines}, each possible projection contributes at most two points to \(H\cap S_f\). Since the possible projections lie in \(L\cap\Sigma\), which has size at most two by Lemma
\ref{lem:surface-lines}, we get \(|H\cap S_f|\le4\). This completes the proof.
\end{proof}

It remains to analyze affine planes \(H\) with \(\dim\pi(H)=2\). Then \(\pi(H)\) is an affine plane in \(\F_q^3\), so there is a nonzero affine linear form \(L_H(X,Y,W)\) such that
\[
   \pi(H)=\{(X,Y,W):L_H(X,Y,W)=0\}.
\]
The map \(\pi|_H:H\to\pi(H)\) is an affine isomorphism. Hence there are affine linear functions \(g_{4,H}\) and \(g_{5,H}\) on \(\pi(H)\) such that the unique point of \(H\) above \((X,Y,W)\in\pi(H)\) is
\[
   (X,Y,W,g_{4,H}(X,Y,W),g_{5,H}(X,Y,W)).
\]
Extend \(g_{4,H}\) and \(g_{5,H}\) arbitrarily to affine linear functions on \(\F_q^3\), and put
\[
   G_H(X,Y,W)=g_{5,H}(X,Y,W)-g_{4,H}(X,Y,W)^2.
\]
Then
\begin{equation}\label{eq:plane-section}
|H\cap S_f|=\bigl|\{(x,y)\in\F_q^2:L_H(x,y,x^2-\alpha y^2)=0,
   f(x,y,x^2-\alpha y^2)=G_H(x,y,x^2-\alpha y^2)\}\bigr|.
\end{equation}
The first equation in \eqref{eq:plane-section} is either an affine line in the \((x,y)\)-plane or an affine conic. It is a line precisely when the coefficient of \(W\) in \(L_H\) is zero. Otherwise, after dividing by that coefficient, it has the form
\[
   x^2-\alpha y^2=L(x,y)
\]
with \(L\) affine linear. Such a conic is either nonsingular or, after a translation, has equation \(X^2-\alpha Y^2=0\). In the singular case it has exactly one \(\F_q\)-point, because \(\alpha\) is a nonsquare.

The following interpolation lemma is the algebraic input of the construction.

\begin{lemma}\label{lem:interpolation}
Let \(r\ge 3\) and \(q>2r+1\). Let \(\Gamma\subset \mathbb A^2_{x,y}\) be either an affine line or a nonsingular affine conic \(x^2-\alpha y^2=L(x,y)\), where \(L\) is affine linear. Let \(R_\Gamma\) be the space of functions on \(\Gamma(\mathbb F_q)\) obtained by restricting
\[
(x,y)\longmapsto f(x,y,x^2-\alpha y^2),\qquad f\in\mathcal F_r .
\]
Then \emph{(i)} in the line case, after choosing an affine parameter \(s\) on \(\Gamma\), one has \(R_\Gamma=\mathbb F_q[s]_{\le 2r}\); and \emph{(ii)} in the nonsingular conic case, one has \(\dim R_\Gamma=2r+1\). In particular, \(\dim R_\Gamma=2r+1\) in both cases. Moreover, evaluation at any \(m\le 2r+1\) distinct points of \(\Gamma(\mathbb F_q)\) gives a surjective map
\[
R_\Gamma\longrightarrow \mathbb F_q^m .
\]
\end{lemma}
\begin{proof}
First suppose that \(\Gamma\) is a line. Choose an affine coordinate \(s\) on  \(\Gamma\). Then \(x=x(s)\) and \(y=y(s)\) are affine linear in \(s\), not both constant. After an affine change of coordinate, we may assume that \(s\) is the restriction of an affine linear function in \(x,y\). Put
\[
Q(s)=x(s)^2-\alpha y(s)^2 .
\]
If \((u,v)\neq (0,0)\) is the direction of \(\Gamma\), then the coefficient of \(s^2\) in \(Q(s)\) is \(u^2-\alpha v^2\), which is nonzero since \(\alpha\) is a nonsquare. Thus \(Q\) has degree two. Hence every restriction of \(f(x,y,x^2-\alpha y^2)\), with \(f\in\mathcal F_r\), has degree at most \(2r\) in \(s\).

Conversely, let \(\ell(X,Y)\) be an affine linear function whose restriction to \(\Gamma\) is \(s\). For \(0\le k\le 2r\), set
\[
(i,j)=
\begin{cases}
(0,k/2),& k\ \text{even},\\
(1,(k-1)/2),& k\ \text{odd}.
\end{cases}
\]
Then \(i+j\le r\), and \(\ell(X,Y)^iW^j\) has total degree at most \(r\). Its restriction to \(\Gamma\) is \(s^iQ(s)^j\), which has degree exactly \(k\) and nonzero leading coefficient. These restrictions therefore form, when ordered by degree, a triangular spanning family for \(\mathbb F_q[s]_{\le 2r}\). This proves (i), and the asserted surjectivity in the line case follows from ordinary Vandermonde interpolation.

Now suppose that \(\Gamma\) is a nonsingular conic. On \(\Gamma\) we have \(W=L(X,Y)\), so \(R_\Gamma\) is the space of restrictions to \(\Gamma\) of affine polynomials in \(X,Y\) of degree at most \(r\). After translating \(X,Y\) over \(\mathbb F_q\), the equation becomes
\[
X^2-\alpha Y^2=c,
\]
with \(c\neq 0\). Its projective closure is
\[
\bar\Gamma:\quad X^2-\alpha Y^2=cZ^2 .
\]
There is no \(\mathbb F_q\)-point on \(\bar\Gamma\) with \(Z=0\), since \(X^2-\alpha Y^2=0\) has no nonzero solution over \(\mathbb F_q\). On the other hand, \(\bar\Gamma\) has an affine \(\mathbb F_q\)-point: if
\(\beta^2=\alpha\) in \(\mathbb F_{q^2}\), then
\[
X^2-\alpha Y^2=N_{\mathbb F_{q^2}/\mathbb F_q}(X+\beta Y),
\]
and the norm map is onto \(\mathbb F_q^\ast\).

Thus \(\bar\Gamma\) is a nonsingular projective conic with an \(\mathbb F_q\)-point, and hence is \(\mathbb F_q\)-isomorphic to \(\mathbb P^1\). Since \(\bar\Gamma\) has no \(\mathbb F_q\)-point at infinity, this isomorphism identifies \(\mathbb P^1(\mathbb F_q)\) with \(\Gamma(\mathbb F_q)\). Under a parametrization
\[
\phi:\mathbb P^1\longrightarrow \bar\Gamma,
\]
the coordinate functions \(X,Y,Z\) pull back to three linearly independent
binary quadratic forms, hence span
\[
H^0(\mathbb P^1,\mathcal O_{\mathbb P^1}(2)).
\]
It follows that homogeneous plane polynomials of degree \(r\) pull back to all binary forms of degree \(2r\): indeed, the monomials \(s^{2r-k}t^k\), \(0\le k\le 2r\), are generated by products of \(r\) elements from \(\langle s^2,st,t^2\rangle\).

Passing from homogeneous to affine polynomials causes no loss on \(\Gamma(\mathbb F_q)\). Since \(Z\neq 0\) at every point of \(\Gamma(\mathbb F_q)\), division by \(Z^r\) is an invertible pointwise rescaling. Therefore the restrictions of affine polynomials of degree at most \(r\) give, after this nonzero rescaling, the same function space as
\[
H^0(\mathbb P^1,\mathcal O_{\mathbb P^1}(2r)).
\]
A nonzero binary form of degree \(2r\) has at most \(2r\) zeros on \(\mathbb P^1\); since \(q>2r+1\), it cannot vanish on all of \(\mathbb P^1(\mathbb F_q)\). Hence this function space has dimension \(2r+1\).

Finally, let \(P_1,\ldots,P_m\) be distinct points of \(\Gamma(\mathbb F_q)\), with \(m\le 2r+1\). After the parametrization by \(\mathbb P^1\), and after the same nonzero pointwise rescaling, evaluation at these points is evaluation of binary forms of degree \(2r\) at \(m\) distinct points of \(\mathbb P^1(\mathbb F_q)\). In an affine coordinate this is ordinary Vandermonde interpolation; if one of the points is infinity, prescribe the
leading coefficient and interpolate the remaining values by the lower coefficients. Thus the evaluation map is surjective. This proves (ii) and the
lemma.
\end{proof}
We shall use the following immediate consequence. Since the restriction map \(\mathcal F_r\to\mathcal R_\Gamma\) is linear and surjective, the restriction of a uniformly chosen \(f\in\mathcal F_r\) is uniformly distributed on \(\mathcal R_\Gamma\). Hence agreement with any fixed element of \(\mathcal R_\Gamma\) has probability \(q^{-(2r+1)}\), and prescribed values at any \(m\le 2r+1\) distinct points of \(\Gamma(\F_q)\) have
probability \(q^{-m}\). For \(f\in\mathcal F_r\), let \(\mathcal E_f\) be the family of affine planes \(H\) with \(\dim\pi(H)=2\) for which the first equation in \eqref{eq:plane-section} defines a line or a nonsingular conic, and
\begin{equation}\label{eq:identity}
   f(x,y,x^2-\alpha y^2)=G_H(x,y,x^2-\alpha y^2)
\end{equation}
holds identically in the restriction space of Lemma \ref{lem:interpolation}. Equivalently, in the line case the two sides agree as polynomials in an affine parameter, and in the nonsingular conic case their pullbacks to \(\mathbb P^1\) agree as sections of \(\mathcal O_{\mathbb P^1}(2r)\) after homogenization. In particular, \eqref{eq:identity} holds at every \(\F_q\)-point of the curve. Singular conics are not included in \(\mathcal E_f\), since they have only one \(\F_q\)-point.

This definition is compatible with Lemma \ref{lem:interpolation}. In the line case, \(G_H(x,y,x^2-\alpha y^2)\) has degree at most four in the line parameter, and thus belongs to the restriction space because \(r\ge 3\). In the conic case, \(W=L(X,Y)\) on the conic, so \(G_H\) restricts to a polynomial of degree at most two in \(x,y\), again belonging to the restriction space.

\begin{lemma}\label{lem:section-bound}
Let \(f\in\mathcal F_r\). If an affine plane \(H\subset\F_q^5\) is neither in \(\mathcal E_f\) nor equal to one of the planes \(Q_{x,y}\) from Lemma \ref{lem:small-projection}\textup{(i)}, then
\[
   |H\cap S_f|\le 2r.
\]
\end{lemma}

\begin{proof}
If \(\dim\pi(H)=0\), then either \(H\cap S_f=\emptyset\), or Lemma \ref{lem:small-projection}(i) gives \(H=Q_{x,y}\) for some \(x,y\in\F_q\). The latter case is excluded. If \(\dim\pi(H)=1\), then Lemma \ref{lem:small-projection}(ii) gives \(|H\cap S_f|\le 4\le 2r\).

Assume \(\dim\pi(H)=2\). By \eqref{eq:plane-section}, the intersection is governed by a line or conic \(\Gamma\) in the \((x,y)\)-plane together with the equation \eqref{eq:identity}. In the line case, the difference between the two sides is a univariate polynomial of degree at most \(2r\). Since \(H\notin\mathcal E_f\), the difference is not identically zero, and hence has at most \(2r\) zeros. In the nonsingular conic case, Lemma \ref{lem:interpolation}(ii) identifies the relevant restrictions with degree-\(2r\) functions on \(\mathbb P^1\). Again \(H\notin\mathcal E_f\) means that the difference is nonzero in that space, and therefore it has at most \(2r\) zeros on \(\Gamma(\F_q)\). In the singular conic case there is at most one \(\F_q\)-point. These cases exhaust all possibilities.
\end{proof}

\begin{lemma}\label{lem:expected-exceptional}
If \(f\) is chosen uniformly from \(\mathcal F_r\), then
\[
   \mathbb E[|\mathcal E_f|]=O_r(q^2).
\]
\end{lemma}

\begin{proof}
There are \(O(q^8)\) affine planes \(H\) with \(\dim\pi(H)=2\) for which the curve in \eqref{eq:plane-section} is a line. Indeed, the projected plane is the inverse image, under \((X,Y,W)\mapsto(X,Y)\), of an affine line in the \((X,Y)\)-plane, giving \(O(q^2)\) choices. Once this projected plane is fixed, \(H\) is the graph of an affine map from a two-dimensional affine space to the two-dimensional kernel of \(\pi\), giving \(q^6\) choices. For each such \(H\), Lemma \ref{lem:interpolation}(i) gives probability \(q^{-(2r+1)}\) that \(H\in\mathcal E_f\). These planes contribute \(O(q^{7-2r})\) in expectation.

There are \(O(q^9)\) remaining affine planes with \(\dim\pi(H)=2\): an affine plane in \(\F_q^5\) is determined by a two-dimensional direction subspace, of which there are \(O(q^6)\), and one of \(q^3\) parallel translates. If the corresponding conic is nonsingular, Lemma \ref{lem:interpolation}(ii) gives probability \(q^{-(2r+1)}\) that \(H\in\mathcal E_f\); if it is singular, \(H\) is not included in \(\mathcal E_f\). These planes contribute \(O(q^{8-2r})\). Since \(r\ge 3\), the sum of the two contributions is \(O_r(q^2)\).
\end{proof}

For an affine plane \(H\) with \(\dim\pi(H)=2\), write
\[
   \Gamma_H=\{(x,y)\in\F_q^2:L_H(x,y,x^2-\alpha y^2)=0\}.
\]

\begin{lemma}\label{lem:many-large-sections}
There is an absolute constant \(c>0\) such that, for all sufficiently large \(q\), at least \(cq^9\) affine planes \(H\) with \(\dim\pi(H)=2\) satisfy
\[
   |\Gamma_H|\ge q/2.
\]
\end{lemma}

\begin{proof}
Let
\[
   Z=\{(x,y,x^2-\alpha y^2,u,v):x,y,u,v\in\F_q\}\subset\F_q^5.
\]
Then \(|Z|=q^4\). If \(\dim\pi(H)=2\), the map \(\pi|_H\) is an affine isomorphism from \(H\) to its projected plane. Hence points of \(H\cap Z\) are in bijection with pairs \((x,y)\in\Gamma_H\), and
\begin{equation}\label{eq:Z-section}
   |H\cap Z|=|\Gamma_H|.
\end{equation}

We count incidences \((z,H)\) with \(z\in Z\), \(z\in H\), and \(\dim\pi(H)=2\). Fix \(z\in Z\). An affine plane through \(z\) is determined by its two-dimensional direction subspace. The number of such direction subspaces in \(\F_q^5\) is \((1+o(1))q^6\). Among them, the direction subspaces whose image under the linear part of \(\pi\) has dimension less than two form only \(O(q^4)\) choices. Indeed, image dimension zero gives one choice, and image dimension one is obtained by first choosing a line in \(\F_q^3\), in \(O(q^2)\) ways, and then choosing a two-dimensional subspace inside its three-dimensional inverse image, in \(O(q^2)\) ways. Thus each \(z\in Z\) is incident with \((1+o(1))q^6\) affine planes satisfying \(\dim\pi(H)=2\). Using \eqref{eq:Z-section}, we obtain
\[
   \sum_H |\Gamma_H|=(1+o(1))q^{10},
\]
where the sum is over all affine planes \(H\) with \(\dim\pi(H)=2\).

The number of such planes is at most \((1+o(1))q^9\). Also \(|\Gamma_H|\le 2q\) for every \(H\), since \(\Gamma_H\) is either an affine line or a plane conic. Let \(M\) be the number of summands with \(|\Gamma_H|\ge q/2\). If \(M<cq^9\), then
\[
   \sum_H |\Gamma_H|\le 2qM+(q/2)(1+o(1))q^9
      \le (2c+1/2+o(1))q^{10}.
\]
Choosing any fixed \(c<1/4\) gives a contradiction for all sufficiently large \(q\).
\end{proof}

\begin{lemma}\label{prop:algebraic-set}
Let \(b\ge 4\), and put \(r=\max\{3,\lceil b/2\rceil\}\). There are constants \(C=C_b\) and \(\delta=\delta_b>0\) such that, for every sufficiently large odd prime power \(q\), there is a set \(S\subset\F_q^5\) and a family \(\mathcal D_0\) of affine planes satisfying:
\begin{enumerate}
\item[\textup{(i)}] \(|S|=q^3\);
\item[\textup{(ii)}] no three points of \(S\) are collinear;
\item[\textup{(iii)}] \(|\mathcal D_0|\le Cq^2\);
\item[\textup{(iv)}] every affine plane not in \(\mathcal D_0\) meets \(S\) in at most \(2r\) points;
\item[\textup{(v)}] at least \(\delta q^9\) different \(b\)-element subsets of \(S\) are contained in affine planes not belonging to \(\mathcal D_0\).
\end{enumerate}
\end{lemma}

\begin{proof}
Choose \(f\in\mathcal F_r\) uniformly and set \(S=S_f\). By Lemma \ref{lem:S-no-three}, the set \(S_f\) always satisfies (ii). Let \(M_f=|\mathcal E_f|\). By Lemma \ref{lem:expected-exceptional},
\[
   \mathbb E[M_f] \le C_0q^2
\]
for a constant \(C_0=C_0(b)\).

We next count coplanar \(b\)-element subsets. Let \(H\) be an affine plane with \(\dim\pi(H)=2\) and \(|\Gamma_H|\ge q/2\). For \((x,y)\in\Gamma_H\), let
\[
   p_H(x,y)=(x,y,x^2-\alpha y^2,
      g_{4,H}(x,y,x^2-\alpha y^2),g_{5,H}(x,y,x^2-\alpha y^2))\in H,
\]
and put
\[
   W_H=\{p_H(x,y):(x,y)\in\Gamma_H\}.
\]
Then \(|W_H|=|\Gamma_H|\ge q/2\). A point \(p_H(x,y)\) belongs to \(S_f\) if and only if
\begin{equation}\label{eq:point-condition}
   f(x,y,x^2-\alpha y^2)=G_H(x,y,x^2-\alpha y^2).
\end{equation}
Since \(|\Gamma_H|\ge q/2\), the curve \(\Gamma_H\) is not a singular conic for sufficiently large \(q\). Thus \(\Gamma_H\) is either a line or a nonsingular conic, and Lemma \ref{lem:interpolation} applies.

Fix \(b\) distinct points of \(W_H\), corresponding to distinct points \((x_1,y_1),\dots,(x_b,y_b)\) of \(\Gamma_H\). The condition that all these points lie in \(S_f\) is the system of \(b\) linear equations
\[
   f(x_i,y_i,x_i^2-\alpha y_i^2)=G_H(x_i,y_i,x_i^2-\alpha y_i^2),
   \qquad 1\le i\le b.
\]
Because \(b\le 2r\), Lemma \ref{lem:interpolation} gives probability \(q^{-b}\) for this event. The event \(H\in\mathcal E_f\) is the stronger event that the same equality holds identically in the restriction space, and hence it implies that all chosen points lie in \(S_f\). Lemma \ref{lem:interpolation} gives probability \(q^{-(2r+1)}\) for this identity. Therefore
\[
   \Pr[\text{the chosen }b\text{ points lie in }S_f\text{ and }H\notin\mathcal E_f]
      =q^{-b}-q^{-(2r+1)}\ge \tfrac12 q^{-b}
\]
for all sufficiently large \(q\).

Because \(S_f\) contains no three collinear points and \(b\ge 4\), any \(b\)-element subset of \(S_f\) contained in a plane determines that plane uniquely: any three of its points are noncollinear and span the plane. Thus the following count of pairs \((H,B)\) is also a count of distinct \(b\)-element subsets. By Lemma \ref{lem:many-large-sections}, there are at least \(cq^9\) affine planes \(H\) with \(\dim\pi(H)=2\) and \(|\Gamma_H|\ge q/2\). Each such plane has \(\Omega_b(q^b)\) choices of a \(b\)-element subset of \(W_H\). If \(X_f\) denotes the number of \(b\)-element subsets of \(S_f\) that lie on an affine plane \(H\notin\mathcal E_f\) with \(\dim\pi(H)=2\) and \(|\Gamma_H|\ge q/2\), the preceding paragraph gives
\[
   \mathbb E[ X_f] \ge c_1q^9
\]
for a constant \(c_1=c_1(b)>0\).

There are \(O(q^9)\) affine planes with \(\dim\pi(H)=2\). By Lemma \ref{lem:section-bound}, each plane counted by \(X_f\) contains at most \(2r\) points of \(S_f\), and hence only \(O_b(1)\) relevant \(b\)-subsets. Thus
\[
   X_f\le C_1q^9
\]
deterministically, for some constant \(C_1=C_1(b)\). Choose \(\lambda>0\) so small that \(c_1-\lambda C_0>c_1/2\). Then
\[
   \mathbb E\bigl[X_f-\lambda q^7M_f\bigr]
      \ge (c_1/2)q^9.
\]
Hence there exists a polynomial \(f\) such that
\[
   X_f-\lambda q^7M_f\ge (c_1/2)q^9.
\]
For this \(f\), one has \(X_f\ge(c_1/2)q^9\), and the deterministic upper bound on \(X_f\) gives
\[
   \lambda q^7M_f\le X_f-(c_1/2)q^9\le C_1q^9.
\]
Thus \(M_f\le C_2q^2\) for a constant \(C_2=C_2(b)\).

Fix this polynomial \(f\). Let \(\mathcal D_0\) consist of the \(q^2\) affine planes \(Q_{x,y}\) from Lemma \ref{lem:small-projection}(i), together with the planes in \(\mathcal E_f\). Then \(|\mathcal D_0|=O_b(q^2)\). Lemma \ref{lem:section-bound} shows that every affine plane outside \(\mathcal D_0\) contains at most \(2r\) points of \(S_f\). Finally, the subsets counted by \(X_f\) give at least \(\delta q^9\) coplanar \(b\)-element subsets on planes outside \(\mathcal D_0\), after decreasing \(\delta\) if necessary.
\end{proof}

\section{The incidence graph}

This section turns the geometric data from Lemma~\ref{prop:algebraic-set} into the graph-theoretic lower bound. We embed the point set into projective space, use a polarity to convert points into hyperplanes, delete the configurations forced by exceptional planes, and then prove that the resulting incidence graph is \(K_{3,2r+1}\)-free while still containing \(\Omega_b(q^9)\) copies of \(K_{b,b}\).

We pass to projective space using the standard affine chart \(\F_q^5\subset\PG(5,q)\), and we again write \(S\) for the image of the affine set. Let \(\mathcal D_0\) be the affine family supplied by Lemma~\ref{prop:algebraic-set}, and let \(\mathcal D\) be the set of projective closures of the planes in \(\mathcal D_0\). Then \(|\mathcal D|=O_b(q^2)\).

The projective version has the following properties. First, no three points of \(S\) are collinear in \(\PG(5,q)\): if three affine points were collinear projectively, their projective line would not be contained in the hyperplane at infinity, and its affine part would be an affine line containing the same three points. Second, every projective plane outside \(\mathcal D\) contains at most \(2r\) points of \(S\). Indeed, a projective plane contained in the hyperplane at infinity contains no point of \(S\); otherwise its affine part is an affine plane in \(\F_q^5\), and if its projective closure is outside \(\mathcal D\), then that affine plane is outside \(\mathcal D_0\). Third, at least \(\delta q^9\) \(b\)-element subsets of \(S\) are contained in projective planes outside \(\mathcal D\), namely the closures of the affine planes counted in Lemma \ref{prop:algebraic-set}(v).

Fix a nondegenerate symmetric bilinear form on the underlying six-dimensional vector space. For a projective subspace \(U\subset\PG(5,q)\), write \(U^\perp\) for the projectivization of the vector-space orthogonal complement. This polarity satisfies \((U^\perp)^\perp=U\) and reverses containment. If \(U\) has projective dimension \(d\), then \(U^\perp\) has projective dimension \(4-d\). In particular, points correspond to hyperplanes, lines to 3-spaces, and planes to planes.

Let \(T\) be a uniformly random projective automorphism of \(\PG(5,q)\). Define a family of hyperplanes
\[
   \mathcal H=\{T(P)^\perp:P\in S\}.
\]
The map \(P\mapsto T(P)^\perp\) is injective, since \(T\) is injective and polarity is a bijection between points and hyperplanes.

Let \(\mathcal P_T\) be the family of projective planes contained in more than \(2r\) hyperplanes of \(\mathcal H\). If \(\Pi\in\mathcal P_T\), then there are more than \(2r\) points \(P\in S\) such that
\[
   \Pi\subset T(P)^\perp.
\]
By containment reversal, this is equivalent to \(T(P)\in\Pi^\perp\). Hence the projective plane \(T^{-1}(\Pi^\perp)\) contains more than \(2r\) points of \(S\), and therefore belongs to \(\mathcal D\). Denoting this plane by \(Q\), we have
\[
   Q=T^{-1}(\Pi^\perp),\qquad \Pi=T(Q)^\perp.
\]
Thus
\begin{equation}\label{eq:bad-planes}
   \mathcal P_T\subset \{T(Q)^\perp:Q\in\mathcal D\},\qquad |\mathcal P_T|=O_b(q^2).
\end{equation}
Moreover, no three hyperplanes in \(\mathcal H\) have a projective 3-space as their intersection. Indeed, suppose that
\[
   T(P_1)^\perp,\quad T(P_2)^\perp,\quad T(P_3)^\perp
\]
have an intersection of projective dimension \(3\). The orthogonal complement of that intersection is then a projective line containing the three points \(T(P_1)\), \(T(P_2)\), and \(T(P_3)\). Applying \(T^{-1}\) gives a projective line containing \(P_1\), \(P_2\), and \(P_3\), contradicting the fact that no three points of \(S\) are collinear.

Delete from \(\mathcal H\) every hyperplane containing a plane from \(\mathcal D\), and denote the remaining family by \(\mathcal H'\). Delete from \(S\) every point lying on a plane in \(\mathcal P_T\), and denote the remaining set by \(S'\). Define the bipartite incidence graph
\[
   G_T=(X_T\sqcup Y_T,E_T),\qquad X_T=S',\qquad Y_T=\mathcal H',
\]
where \((P,H)\in E_T\) if and only if \(P\in H\).

For points \(P_1,\dots,P_k\in X_T\), their common neighborhood in \(G_T\) is the set of hyperplanes in \(Y_T\) containing \(\langle P_1,\dots,P_k\rangle\). Dually, for hyperplanes \(H_1,\dots,H_k\in Y_T\), their common neighborhood is \(S'\cap H_1\cap\cdots\cap H_k\).

\begin{lemma}\label{lem:Kfree}
The graph \(G_T\) is \(K_{3,2r+1}\)-free.
\end{lemma}

\begin{proof}
Since \(G_T\) is bipartite with parts \(X_T=S'\) and \(Y_T=\mathcal H'\), there are two possible positions for the part of size three.

First assume that \(P_1,P_2,P_3\in X_T\) have at least \(2r+1\) common neighbors in \(Y_T\). Since no three points of \(S\) are collinear, these three points span a projective plane
\[
   \Pi=\langle P_1,P_2,P_3\rangle.
\]
Every common neighbor is a hyperplane containing \(\Pi\). Thus \(\Pi\) is contained in at least \(2r+1\) hyperplanes of \(\mathcal H\), so \(\Pi\in\mathcal P_T\). Each of \(P_1,P_2,P_3\) lies on \(\Pi\), and hence all three points would have been deleted in the formation of \(S'\), a contradiction.

Now suppose that \(H_1,H_2,H_3\in Y_T\) have at least \(2r+1\) common neighbors in \(X_T\). Write \(H_i=T(P_i)^\perp\) with \(P_i\in S\). The points \(T(P_1),T(P_2),T(P_3)\) are not collinear, since \(S\) contains no three collinear points. Hence they span a projective plane, and
\[
   J=H_1\cap H_2\cap H_3
     =\langle T(P_1),T(P_2),T(P_3)\rangle^\perp
\]
is a projective plane. This plane contains at least \(2r+1\) points of \(S\), namely the common neighbors of \(H_1,H_2,H_3\) in \(S'\). Therefore \(J\in\mathcal D\). Since each \(H_i\) contains \(J\), each \(H_i\) would have been deleted in the passage from \(\mathcal H\) to \(\mathcal H'\), a contradiction.
\end{proof}

\begin{lemma}\label{lem:many-Kbb}
For some choice of the projective automorphism \(T\), the graph \(G_T\) contains \(\Omega_b(q^9)\) ordered copies of \(K_{b,b}\) whose first side lies in \(S'\) and whose second side lies in \(\mathcal H'\).
\end{lemma}

\begin{proof}
 
Let \(\mathcal B\) be the family of all \(b\)-element subsets \(A\subset S\)
which are contained in a projective plane outside \(\mathcal D\). Since \(S\) contains no three collinear points and \(b\ge 4\), this plane is unique for every \(A\in\mathcal B\); denote it by \(\Pi_A\). Lemma~\ref{prop:algebraic-set} gives
\[
|\mathcal B|\ge \delta q^9 .
\]
All probabilities and expectations are over the random projective automorphism \(T\). Let \(G_T^0\) be the bipartite incidence graph with parts \(S\) and \(\mathcal H\), before deletion. For \(A,C\in\mathcal B\), consider the ordered pair
\[
\left(A,\{T(R)^\perp:R\in C\}\right)
\]
in \(G_T^0\). For \(P\in A\) and \(R\in C\), the incidence \(P\in T(R)^\perp\) is equivalent, by symmetry of the polarity, to \(T(R)\in P^\perp\). Hence all incidences hold if and only if
\[
T(\Pi_C)\subseteq \bigcap_{P\in A}P^\perp=\Pi_A^\perp .
\]
Both sides are projective planes, so the inclusion is equality. Applying \(\perp\), we obtain the equivalent condition
\begin{equation}\label{eq:incidence-plane-condition}
\Pi_A=T(\Pi_C)^\perp .
\end{equation}
For fixed \(A,C\), the plane \(T(\Pi_C)^\perp\) is uniformly distributed among all projective planes of \(\mathrm{PG}(5,q)\). Since the number of projective planes is \((1+o(1))q^9\), condition \eqref{eq:incidence-plane-condition} has probability \((1+o(1))q^{-9}\). Summing over the \(|\mathcal B|^2=\Omega_b(q^{18})\) ordered pairs \((A,C)\), the expected number of such ordered \(K_{b,b}\) copies in \(G_T^0\) is \(\Omega_b(q^9)\).

It remains to show that the deletion destroys only \(O_b(q^8)\) of these copies in expectation. A copy destroyed by point deletion is assigned to one deleted vertex on its first side. Thus the expected loss from point deletion is at most
\[
\sum_{P\in S}\Pr[P\notin S']M_P,
\]
where \(M_P\) is the maximum, over all \(T\), of the number of copies counted above whose first side contains \(P\). We claim that \(M_P=O_b(q^6)\). Indeed, choose two further points of \(S\) on the first side; there are \(O(q^6)\) choices. Together with \(P\), they span \(\Pi_A\). If \(A\in\mathcal B\), then \(\Pi_A\notin \mathcal D\), so \(|\Pi_A\cap S|\le 2r\), and the remaining points of \(A\) have \(O_b(1)\) choices. Equation \eqref{eq:incidence-plane-condition} then determines
\[
\Pi_C=T^{-1}(\Pi_A^\perp).
\]
For a counted copy one must have \(C\in\mathcal B\), so \(\Pi_C\notin \mathcal D\) and \(|\Pi_C\cap S|\le 2r\); hence \(C\) has \(O_b(1)\) choices. Thus \(M_P=O_b(q^6)\).

For fixed \(P\in S\), we have \(\Pr[P\notin S']=O_b(q^{-1})\). Indeed, if \(P\notin S'\), then \(P\in\Pi\) for some \(\Pi\in \mathcal P_T\). By \eqref{eq:bad-planes}, this implies \(P\in T(Q)^\perp\) for some \(Q\in \mathcal D\). For fixed \(Q\), the plane \(T(Q)^\perp\) is uniformly distributed among all projective planes; the probability that it contains \(P\) is \((1+o(1))q^{-3}\). Since
\(|\mathcal D|=O_b(q^2)\), the union bound gives \(\Pr[P\notin S']=O_b(q^{-1})\). Therefore the expected loss from point deletion is
\[
O_b(q^3\cdot q^{-1}\cdot q^6)=O_b(q^8).
\]

The hyperplane deletion estimate is similar. Assign each copy destroyed by hyperplane deletion to one deleted vertex on its second side. A hyperplane vertex of \(\mathcal H\) has the form \(T(P)^\perp\), with \(P\in S\). For every \(T\), the number of counted copies whose second side contains this vertex is \(O_b(q^6)\): choose two further points of \(S\) in the indexing set \(C\), determine \(\Pi_C\), complete \(C\) in \(O_b(1)\) ways using
\(|\Pi_C\cap S|\le 2r\), and then determine
\[
\Pi_A=T(\Pi_C)^\perp
\]
from \eqref{eq:incidence-plane-condition}; if this gives a counted copy, then \(\Pi_A\notin \mathcal D\), so \(A\) has \(O_b(1)\) choices.

For fixed \(P\in S\), the event \(T(P)^\perp\notin \mathcal H'\) can occur only if \(T(P)^\perp\) contains some plane \(Q\in \mathcal D\), equivalently \(T(P)\in Q^\perp\). For fixed \(Q\), the plane \(Q^\perp\) contains \((1+o(1))q^2\) of the \((1+o(1))q^5\) points of \(\mathrm{PG}(5,q)\), so this probability is \((1+o(1))q^{-3}\). Taking the union bound over \(Q\in \mathcal D\) gives
\(\Pr[T(P)^\perp\notin \mathcal H']=O_b(q^{-1})\). Hence the expected loss from hyperplane deletion is also \(O_b(q^8)\).

Thus the expected number of ordered copies remaining after deletion is
\[
\Omega_b(q^9)-O_b(q^8)=\Omega_b(q^9).
\]
Consequently some choice of \(T\) has at least this many copies.
\end{proof}

\begin{lemma}\label{prop:lower-bound-graph}
Let \(3<a\le b\), and put \(r=\max\{3,\lceil b/2\rceil\}\). For every sufficiently large \(N\), there is an \(N\)-vertex bipartite \(K_{3,2r+1}\)-free graph containing \(\Omega_{a,b}(N^3)\) ordered copies \((A,B)\) of \(K_{a,b}\) with \(A\) in one bipartition class and \(B\) in the other.
\end{lemma}

\begin{proof}
Choose \(q\) and construct \(G_T\) as above. By Lemma \ref{lem:Kfree}, the graph is \(K_{3,2r+1}\)-free. By Lemma \ref{lem:many-Kbb}, for some \(T\) it contains \(\Omega_b(q^9)\) ordered copies \((U,V)\) of \(K_{b,b}\) with \(U\subset X_T\) and \(V\subset Y_T\). Each such copy contains \(\binom ba\) ordered copies \((A,V)\) of \(K_{a,b}\) with \(A\subset U\).

We claim that any fixed ordered copy \((A,B)\), with \(A\subset X_T\) and \(B\subset Y_T\), arises from only \(O_{a,b}(1)\) of the ordered \(K_{b,b}\) copies counted above. Since \(|B|=b\), necessarily \(V=B\). To choose \(U\), one must add \(b-a\) vertices of \(X_T\) to \(A\), all adjacent to every vertex of \(B\). Because \(b\ge a>3\), choose three vertices of \(B\). Their common neighborhood has size at most \(2r\), since \(G_T\) is \(K_{3,2r+1}\)-free. Hence there are at most
\[
   \binom{2r}{b-a}=O_{a,b}(1)
\]
possible choices for \(U\). Thus the \(\Omega_b(q^9)\) incidences between ordered \(K_{b,b}\) copies and the ordered \(K_{a,b}\) copies obtained from them yield \(\Omega_{a,b}(q^9)\) distinct ordered copies of \(K_{a,b}\).

The graph has at most \(2q^3\) vertices. Given sufficiently large \(N\), apply Bertrand's postulate with \(x=(N/16)^{1/3}\). There is a prime \(q\) with \(x\le q\le 2x\); for large \(N\) this prime is odd. Then \(q=\Theta(N^{1/3})\) and \(2q^3\le N\). Adding isolated vertices gives an \(N\)-vertex graph, preserves \(K_{3,2r+1}\)-freeness, and leaves \(\Omega_{a,b}(N^3)\) ordered copies of \(K_{a,b}\).
\end{proof}

\begin{proof}[Proof of Theorem \ref{thm:main}]
Let
\[
r=\max\{3,\lceil b/2\rceil\}.
\]
Then \(\tau(b)=2r+1\). Since \(t\ge \tau(b)>b\) and \(3<a\le b\), the
upper bound follows immediately from \eqref{Mian-bound-general-tran-bi}, with \(s=3\):
\[
\operatorname{ex}(n,K_{a,b},K_{3,t})=O_{a,b,t}(n^3).
\]

It remains to establish the matching lower bound. By Lemma~\ref{prop:lower-bound-graph}, for every sufficiently large \(n\) there exists an \(n\)-vertex bipartite graph \(G\) which is \(K_{3,2r+1}\)-free and contains \(\Omega_{a,b}(n^3)\) ordered copies \((A,B)\) of \(K_{a,b}\), with \(A\) and \(B\) lying in the two prescribed bipartition classes. Since \(t\ge 2r+1\), every copy of \(K_{3,t}\) contains a copy of \(K_{3,2r+1}\). Thus \(G\) is \(K_{3,t}\)-free.

Finally, the distinction between these ordered copies and the usual copies of \(K_{a,b}\) affects the count by at most a constant factor depending only on \(a\) and \(b\). Therefore \(G\) contains \(\Omega_{a,b}(n^3)\) copies of \(K_{a,b}\) in the usual sense, and hence
\[
\operatorname{ex}(n,K_{a,b},K_{3,t})=\Omega_{a,b}(n^3).
\]
Combining the upper and lower bounds gives
\[
\operatorname{ex}(n,K_{a,b},K_{3,t})=\Theta_{a,b,t}(n^3),
\]
as claimed.
\end{proof}

\section{Concluding remarks}

Theorem \ref{thm:main} gives an explicit threshold for the \(K_{3,t}\) construction. The bound \(t\ge \tau(b)\) comes from the direct analysis of line and conic sections of the algebraic point set. In the odd case \(b\ge 5\), the result gives \(b+2\). This suggests the following problem.
\begin{problem}
For odd \(b\ge 5\) and fixed \(3<a\le b\), is
\[
   \ex(n,K_{a,b},K_{3,b+1})=\Theta_{a,b}(n^3)?
\]
\end{problem}

\section*{Declaration on the use of AI}

The authors used generative AI tools to assist in discussing proof strategies, checking proofs, and improving exposition.

\nocite{*}
\bibliographystyle{plain}
\bibliography{refs}

\end{document}